\documentclass[11pt]{amsart}%
\usepackage{amscd,amsmath,latexsym,amsthm,amsfonts,amssymb,graphicx,color,geometry}
\usepackage{enumerate}
\usepackage[pdftex,plainpages=false,pdfpagelabels,
colorlinks=true,citecolor=blue,linkcolor=black,urlcolor=black,
filecolor=black,bookmarksopen=true,unicode=true]{hyperref}
\usepackage[norefs,nocites]{refcheck}
\usepackage{amsmath}
\usepackage{amsfonts}
\usepackage{amssymb}
\usepackage[T1]{fontenc}
\usepackage{graphicx}
\usepackage{xcolor}
\usepackage{colortbl}
\usepackage{float}
\usepackage{framed}
\usepackage{pgf,tikz,pgfplots}
\usepackage{tkz-base}
\usepackage{multicol}
\usepackage{tabularx}
\usepackage{tkz-fct}%
\providecommand{\U}[1]{\protect\rule{.1in}{.1in}}
\tikzset{>=Triangle}
\newtheorem{theorem}{Theorem}[section]
\newtheorem{proposition}[theorem]{Proposition}

\newtheorem{corollary}[theorem]{Corollary}

\newtheorem{definition}[theorem]{Definition}
\numberwithin{equation}{section}
\pgfplotsset{compat=1.17}
\begin{document}
\title[A Quest for Convergence: Exploring Series in Non-Linear Environments]{A Quest for Convergence: Exploring Series in Non-Linear Environments}
\author[Geivison Ribeiro]{Geivison Ribeiro}
\address{Departamento de Matem\'{a}tica \\
Universidade Federal da Para\'{\i}ba \\
58.051-900 - Jo\~{a}o Pessoa, Brazil.}
\email{geivison.ribeiro@academico.ufpb.br}
\thanks{G. Ribeiro is supported by Grant 2022/1962, Para\'{\i}ba State Research
Foundation (FAPESQ)}

\begin{abstract}
This note presents an extension of a result within the concept of
$[\mathcal{S}]$-lineability, originally developed in 2019 by L.
Bernal-Gonz\'{a}lez, J.A. Conejero, M. Murillo-Arcila, and J.B.
Seoane-Sep\'{u}lveda $\left(  \text{see \cite{S lineability}}\right)  $.
Additionally, we provide a characterization in terms of lineability in the
context of complements of unions of closed subspaces in $F$-spaces, and
finally, we present a negative result in both normed spaces and $p$-Banach
spaces. These findings contribute to the understanding of linearity in exotic
settings in topological vector spaces.

\end{abstract}
\maketitle

\section{Introduction and background}

In recent years, the exploration of linearity within unconventional
mathematical contexts has emerged as a notable trend. This pursuit has been
underscored by the introduction of terms like \textit{lineability} and
\textit{spaceability}, originally coined by V.I. Gurary, which gained
prominence following their appearance in the seminal paper by Aron, Gurariy,
and Seoane-Sep\'{u}lveda; we refer the reader to \cite{AGSS}, $\left(
\text{see also \cite{Seoane}}\right)  $. The inception of these concepts
sparked interest among many mathematicians, catalyzing an increase in research
efforts in this direction. According to these terminologies, a subset $A$ of a
vector space $X$ is said to be lineable if $A\cup\left\{  0\right\}  $
encompasses an infinite dimensional vector subspace. Furthermore, in the
context of $X$ being a topological vector space, $A$ is called spaceable if
$A\cup\left\{  0\right\}  $ contains a closed infinite dimensional vector
subspace. A detailed account of lineability can be found in \cite{Aron,
Bernal, Enflo, Pellegrino}.

In 2004, during an Analysis Seminar at Kent State University, Gurariy
introduced the following concept of $\left[  \mathcal{S}\right]  $-lineability
(see \cite{S lineability}): Given a Hausdorff topological vector space
$X\neq\left\{  0\right\}  $ and a vector subspace $\mathcal{S}$ of
$\mathbb{K}^{\mathbb{N}}$ $\left(  \text{where }\mathbb{K=R}\text{ or
}\mathbb{C}\right)  $, a subset $A$ in $X$ is defined as:

\begin{itemize}
\item $\left[  \left(  u_{n}\right)  _{n=1}^{\infty},\mathcal{S}\right]
$-lineable\textbf{ }in $X$ if, for each sequence $\left(  c_{n}\right)
_{n=1}^{\infty}\in\mathcal{S}$, the series $\sum_{n=1}^{\infty}c_{n}u_{n}$
converges in $X$ to a vector of $A\cup\left\{  0\right\}  $.

\item $\left[  \mathcal{S}\right]  $-lineable in $X$ if it is $\left[  \left(
u_{n}\right)  _{n=1}^{\infty},\mathcal{S}\right]  $-lineable for some sequence
$\left(  u_{n}\right)  _{n=1}^{\infty}$ of linearly independent elements in
$X$.
\end{itemize}

As far as we know, the notion of $\left[  \mathcal{S}\right]  $-lineability
has been developed only in \cite{S lineability}.

To avoid trivial or undesirable situations, we shall assume throughout this
paper that a vector space $X$ never collapses to $\left\{  0\right\}  $.
Throughout this note, we will use the term subspace instead of vector subspace.

The symbol $\operatorname*{card}\left(  M\right)  $ will denote the
cardinality of the set $M$, and in particular, we will denote $\aleph_{0}$ as
the cardinality of $\mathbb{N}$.

This note is organized as follows. In Section \ref{section 2}, we consider
that $X$ is an $F$-space and we establish that a set of the form
$X\setminus\bigcup_{i\in I}Y_{i}$ is lineable if and only if it is $\left[
\mathcal{S}\right]  $-lineable for every nontrivial subspace $\mathcal{S}$ of
$\ell_{\infty}$. Recall that an $F$-space is a complete metrizable topological
vector space, and that every $F$-space $X$ supports an $F$-norm $\left\Vert
\cdot\right\Vert _{X}$. Recall also that an $F$-norm on a vector space $X$ is
a function $\left\Vert \cdot\right\Vert \colon X\longrightarrow\lbrack
0,\infty)$ satisfying, for all $x,y\in X$ and $\lambda\in\mathbb{K}$ the
following properties:

\begin{enumerate}
\item $\left\Vert x\right\Vert _{X}=0$ if and only if $x=0$.

\item $\left\Vert \lambda x\right\Vert _{X}\leq\left\Vert x\right\Vert _{X}$
if $\left\vert \lambda\right\vert \leq1$.

\item $\left\Vert \lambda x\right\Vert _{X}\rightarrow0$ if $\left\vert
\lambda\right\vert \rightarrow0$.

\item $\left\Vert x+y\right\Vert _{X}\leq\left\Vert x\right\Vert
_{X}+\left\Vert y\right\Vert _{X}$.
\end{enumerate}

Moving to Section \ref{section 3}, we extend a result in the realm of $\left[
\mathcal{S}\right]  $-lineability, originally attributed to L.
Bernal-Gonz\'{a}lez, J.A. Conejero, M. Murillo-Arcila, and J.B.
Seoane-Sep\'{u}lveda. Lastly, in Section \ref{section 4}, we present negative results.

\section{$\left[  \mathcal{S}\right]  $-lineability in the context of
complements of unions\label{section 2}}

In this section, we aim to characterize the family of complements of unions of
closed subspaces through the notion of $\left[  \mathcal{S}\right]
$-lineability. To this end, let us begin with the following notion as appeared
in \cite{Drewnowski}.

\begin{definition}
We say that a sequence $\left(  u_{n}\right)  _{n=1}^{\infty}$ of elements of
a topological vector space $X$ is topologically linearly independent, if for
each sequence $\left(  c_{n}\right)  _{n=1}^{\infty}\in\mathbb{K}^{\mathbb{N}%
}$ with $\sum_{n=1}^{\infty}c_{n}u_{n}=0$, we have $\left(  c_{n}\right)
_{n=1}^{\infty}=0$.
\end{definition}

Based on this definition, if $\mathcal{S}\neq\left\{  0\right\}  $ is a
subspace of $\mathbb{K}^{\mathbb{N}}$ then we will say that a sequence
$\left(  u_{n}\right)  _{n=1}^{\infty}$ of elements of a topological vector
space $X$ is $\mathcal{S}$-topologically linearly independent in\textbf{ }$X$
(or $\mathcal{S}$-independent) if for each sequence $\left(  c_{n}\right)
_{n=1}^{\infty}\in\mathcal{S}$ with $\sum_{n=1}^{\infty}c_{n}u_{n}=0$, we have
$\left(  c_{n}\right)  _{n=1}^{\infty}=0$.

Within this perspective, still in \cite{Drewnowski}, Drewnowski et al.
demonstrated the following result, which will be a crucial ingredient for the
proof of the theorem \ref{lineab x S_lineab}. It is worth mentioning that this
result establishes a connection between linear $\ell_{\infty}$-independence
and linear independence.

\begin{proposition}
\label{Proposition sequence l-independent} Assume that $\left(  x_{n}\right)
_{n=1}^{\infty}$ is a linearly independent sequence in a Hausdorff topological
vector space $X$. Then there is $\lambda_{n}>0$ such that $\left(  \lambda
_{n}x_{n}\right)  _{n=1}^{\infty}$ is $\ell_{\infty}$-independent.
\end{proposition}

Em \cite{Kitson} , the authors made a notable contribution with the following result:

\begin{theorem}
\cite[Theorem 7.4.2]{Kitson} Let $Z_{r}$ $\left(  r\in\mathbb{N}\right)  $ be
Banach spaces and let $X$ be a Fr\'{e}chet space. Let $T_{r}\colon
Z_{r}\longrightarrow X$ be continuous linear mappings and
$Y=\operatorname*{span}\left(
{\textstyle\bigcup\nolimits_{r=1}^{\infty}}
T_{r}\left(  Z_{r}\right)  \right)  $. If $Y$ is not closed in $X$ then the
complement $X\smallsetminus Y$ is spaceable.
\end{theorem}

At this point, a characterization of spaceability for the complements of
non-enumerable unions of subspaces in the context of $F$-spaces in terms of
lineability has yet to be developed. It is not even known if these complements
are $[\mathcal{S}]$-lineable, and it is precisely for this purpose that we
present the following result:

\begin{theorem}
\label{lineab x S_lineab}Let $X\neq\left\{  0\right\}  $ be an $F$-space and
$\left\{  Y_{i}\right\}  _{i\in I}$ be a family of nontrivial closed subspaces
of $X$. The set $X\setminus\bigcup_{i\in I}Y_{i}$ is lineable if and only if
it is $\left[  \mathcal{S}\right]  $-lineable for every subspace
$\mathcal{S}\neq\left\{  0\right\}  $ of $\ell_{\infty}$.
\end{theorem}

\begin{proof}
Assume that $X\setminus\bigcup_{i\in I}Y_{i}$ is lineable and consider the
quotient map $Q_{i}\colon X\longrightarrow X/Y_{i}$ $\left(  i\in I\right)  $.
Let $E$ be an infinite dimensional subspace of $X$ such that
\[
E\cap\left(  \bigcup\limits_{i\in I}Y_{i}\right)  =\left\{  0\right\}
\text{.}%
\]
Let $\left(  x_{n}\right)  _{n=1}^{\infty}$ be a sequence of elements linearly
independent in $E$ such that
\begin{equation}
\sum\limits_{n=1}^{\infty}\left\Vert x_{n}\right\Vert _{X}<\infty\text{.}
\label{serie da norma de xn convergindo}%
\end{equation}
Fix $i\in I$. Due to the fact that $X\setminus\bigcup_{i\in I}Y_{i}$ is
lineable, we can infer that the sequence $\left(  Q_{i}\left(  x_{n}\right)
\right)  _{n=1}^{\infty}$ is also linearly independent in $X/Y_{i}$. Indeed,
if for some $N\in\mathbb{N}$, we have%
\[
\sum\limits_{n=1}^{N}a_{n}Q_{i}\left(  x_{n}\right)  =0\text{,}%
\]
then%
\[
Q_{i}\left(  \sum\limits_{n=1}^{N}a_{n}x_{n}\right)  =0\text{.}%
\]
Hence%
\[
\sum\limits_{n=1}^{N}a_{n}x_{n}\in Y_{i}\cap E=\left\{  0\right\}  \text{.}%
\]
That is,%
\[
a_{n}=0\text{,}%
\]
which proves the desired linear independence in $X/Y_{i}$. Since $X/Y_{i}$ is
a Hausdorff topological vector space, we can invoke Proposition
\ref{Proposition sequence l-independent} to obtain a sequence of positive real
numbers $\left(  \lambda_{n}\right)  _{n=1}^{\infty}$ such that the sequence
$\left(  Q_{i}\left(  \lambda_{n}x_{n}\right)  \right)  _{n=1}^{\infty}$is
$\ell_{\infty}$-independent in $X/Y_{i}$. We also claim that the sequence
$\left(  \alpha_{n}\right)  _{n=1}^{\infty}$, where $\alpha_{n}:=\frac
{\lambda_{n}}{1+\lambda_{n}}$ is such that $\left(  Q_{i}\left(  \alpha
_{n}x_{n}\right)  \right)  _{n=1}^{\infty}$is $\ell_{\infty}$-independent in
$X/Y_{i}$. Indeed, let $\left(  t_{n}\right)  _{n=1}^{\infty}\in\ell_{\infty}$
be such that
\[
\sum\limits_{n=1}^{\infty}t_{n}Q_{i}\left(  \alpha_{n}x_{n}\right)  =0\text{.}%
\]
Hence,
\begin{equation}
\sum\limits_{n=1}^{\infty}\frac{t_{n}}{1+\lambda_{n}}Q\left(  \lambda_{n}%
x_{n}\right)  =0\text{.} \label{combinacao de tn sobre um mais lambda n}%
\end{equation}
Since
\[
\left(  \frac{t_{n}}{1+\lambda_{n}}\right)  _{n=1}^{\infty}\in\ell_{\infty
}\text{,}%
\]
and $\left(  Q_{i}\left(  \lambda_{n}x_{n}\right)  \right)  _{n=1}^{\infty}$is
$\ell_{\infty}$-independent in $X/Y_{i}$, we get%
\[
\frac{t_{n}}{1+\lambda_{n}}=0\text{ for each }n\in\mathbb{N}\text{,}%
\]
and this entails that
\[
t_{n}=0\text{ for each }n\in\mathbb{N}\text{.}%
\]
Therefore, $\left(  Q_{i}\left(  \alpha_{n}x_{n}\right)  \right)
_{n=1}^{\infty}$ is really $\ell_{\infty}$-independent in $X/Y_{i}$.
Furthermore, since\textbf{ }$\left\vert \alpha_{n}\right\vert \leq1$\textbf{
}for all\textbf{ }$n\in N$\textbf{, }we have
\[
\left\Vert \alpha_{n}x_{n}\right\Vert _{X}\leq\left\Vert x_{n}\right\Vert
_{X}\text{ for all }n\in\mathbb{N}\text{.}%
\]
Since%
\[
\sum\limits_{n=1}^{\infty}\left\Vert x_{n}\right\Vert _{X}<\infty\text{,}%
\]
we obtain
\[
\sum\limits_{n=1}^{\infty}\left\Vert \alpha_{n}x_{n}\right\Vert _{X}\leq
\sum\limits_{n=1}^{\infty}\left\Vert x_{n}\right\Vert _{X}<\infty\text{.}%
\]
Now, let $c=\left(  c_{n}\right)  _{n=1}^{\infty}\in\ell_{\infty}%
\setminus\left\{  0\right\}  $. We claim that the series $\sum\limits_{n=1}%
^{\infty}c_{n}\alpha_{n}x_{n}$ converges in $X$. Indeed, due to the fact that
\[
\left\Vert \left\Vert c\right\Vert _{\ell_{\infty}}^{-1}c_{n}\alpha_{n}%
x_{n}\right\Vert _{X}\leq\left\Vert \alpha_{n}x_{n}\right\Vert _{X}%
<\infty\text{,}%
\]
if we take $\varepsilon>0$, then there exists $n_{0}\in\mathbb{N}$ such that%
\[
\sum\limits_{n=n_{0}+1}^{\infty}\left\Vert \left\Vert c\right\Vert
_{\ell_{\infty}}^{-1}c_{n}\alpha_{n}x_{n}\right\Vert _{X}<\varepsilon\text{.}%
\]
Thus, for $r,s\in\mathbb{N}$ with $r,s\geq n_{0}$, we have
\begin{align*}
\left\Vert \sum\limits_{n=1}^{r}\left\Vert c\right\Vert _{\ell_{\infty}}%
^{-1}c_{n}\alpha_{n}x_{n}-\sum\limits_{n=1}^{s}\left\Vert c\right\Vert
_{\ell_{\infty}}^{-1}c_{n}\alpha_{n}x_{n}\right\Vert _{X}  &  =\left\Vert
\sum\limits_{n=s+1}^{r}\left\Vert c\right\Vert _{\ell_{\infty}}^{-1}%
c_{n}\alpha_{n}x_{n}\right\Vert _{X}\\
&  \overset{s<r}{\leq}\sum\limits_{n=s+1}^{r}\left\Vert \left\Vert
c\right\Vert _{\ell_{\infty}}^{-1}c_{n}\alpha_{n}x_{n}\right\Vert _{X}\\
&  \leq\sum\limits_{n=n_{0}+1}^{\infty}\left\Vert \left\Vert c\right\Vert
_{\ell_{\infty}}^{-1}c_{n}\alpha_{n}x_{n}\right\Vert _{X}\\
&  <\varepsilon\text{.}%
\end{align*}
Since $X$ is complete and Hausdorff, we can conclude that sequence $\left(
\sum_{n=1}^{r}\left\Vert c\right\Vert _{\ell_{\infty}}^{-1}c_{n}\alpha
_{n}x_{n}\right)  _{r=1}^{\infty}$ converges in $X$ for the series%
\[
\sum\limits_{n=1}^{\infty}\left\Vert c\right\Vert _{\ell_{\infty}}^{-1}%
c_{n}\alpha_{n}x_{n}\text{.}%
\]
Hence,%
\[
\sum\limits_{n=1}^{\infty}\left\Vert c\right\Vert _{\ell_{\infty}}^{-1}%
c_{n}\alpha_{n}x_{n}\in X\text{,}%
\]
and consequently,%
\[
\sum\limits_{n=1}^{\infty}c_{n}\alpha_{n}x_{n}=\left\Vert c\right\Vert
_{\ell_{\infty}}\sum\limits_{n=1}^{\infty}\left\Vert c\right\Vert
_{\ell_{\infty}}^{-1}c_{n}\alpha_{n}x_{n}\in X\text{.}%
\]
Our aim from this point forward will be to demonstrate that the series
$\sum_{n=1}^{\infty}c_{n}\alpha_{n}x_{n}$ converges to a vector in $\left(
X\setminus\bigcup_{i\in I}Y_{i}\right)  \cup\left\{  0\right\}  $. To achieve
this, we will assume, by way of contradiction, the existence of $y\in\left(
\bigcup_{i\in I}Y_{i}\right)  \setminus\left\{  0\right\}  $ such that
\begin{equation}
\sum\limits_{n=1}^{\infty}c_{n}\alpha_{n}x_{n}=y\text{.}
\label{The series converging to y}%
\end{equation}
Since $y\in Y_{i}$ for some $i\in I$, we have
\[
Q_{i}\left(
{\textstyle\sum\limits_{n=1}^{\infty}}
c_{n}\alpha_{n}x_{n}\right)  =Q_{i}\left(  y\right)  =0\text{.}%
\]
That is,
\[%
{\textstyle\sum\limits_{n=1}^{\infty}}
c_{n}Q_{i}\left(  \alpha_{n}x_{n}\right)  =0\text{.}%
\]
Since $\left(  Q_{i}\left(  \alpha_{n}x_{n}\right)  \right)  _{n=1}^{\infty}$
is $\ell_{\infty}$-independent in $X/Y_{i}$ and $\left(  c_{n}\right)
_{n=1}^{\infty}\in\ell_{\infty}$, we get
\[
\left(  c_{n}\right)  _{n=1}^{\infty}=0\text{.}%
\]
From $\left(  \text{\ref{The series converging to y}}\right)  $, we obtain
\[
y=0\text{,}%
\]
which leads to a contradiction. Therefore, the series $\sum_{n=1}^{\infty
}c_{n}\alpha_{n}x_{n}$ converges to a vector in $\left(  X\setminus
\bigcup_{i\in I}Y_{i}\right)  \cup\left\{  0\right\}  $. This shows that
$X\setminus\bigcup_{i\in I}Y_{i}$ is $[\ell_{\infty}]$-lineable and by \cite{S
lineability}, we conclude that $X\setminus\bigcup_{i\in I}Y_{i}$ is
$[\mathcal{S}]$-lineable for each subspace $\mathcal{S}$ of $\ell_{\infty}$.
For the proof of the converse, it suffices to take $\mathcal{S}=c_{00}$ and
utilize the fact that $X\setminus Y$ is $\left[  c_{00}\right]  $-lineable if
and only if it is lineable.
\end{proof}

\section{A kind of extension\label{section 3}}

A result due Bernal et al$.$ in \cite{S lineability} states that if $A$ is a
subset $[\left(  x_{n}\right)  _{n=1}^{\infty},\mathcal{S}]$-lineable of a
metrizable vector space $X$, where $\mathcal{S}$ is an infinite dimensional
subspace of $\mathbb{K}^{\mathbb{N}}$ and $\left(  x_{n}\right)
_{n=1}^{\infty}\,$is a basic sequence, (recall that a sequence $\left(
x_{n}\right)  _{n=1}^{\infty}$of a metrizable topological vector space $X$ is
said to be basic whenever every $x\in\operatorname*{span}\left\{  x_{n}%
:n\in\mathbb{N}\right\}  $, the closed linear span of the $x_{n}$'s, can be
uniquely represented as a convergent series $x=\sum_{n=1}^{\infty}c_{n}x_{n}$
), then $A$ is lineable. More precisely, they showed the following:

\begin{proposition}
Assume that $\left(  x_{n}\right)  _{n=1}^{\infty}$ is a basic sequence in
$X$, where $X$ is metrizable, and that $A$ is a $[\left(  x_{n}\right)
_{n=1}^{\infty},\mathcal{S}]$-lineable subset of $X$, where $\mathcal{S}$ is
an infinite dimensional subspace of $\mathbb{K}^{\mathbb{N}}$. Then $A$ is lineable.
\end{proposition}

In this section, we complement this result by removing the assumption of
metrizability and the requirement for the sequence to be basic.

\begin{theorem}
Let $X\neq\left\{  0\right\}  $ be a Hausdorff topological vector space and
$A$ be a nonvoid subset of $X$. If $\left(  x_{n}\right)  _{n=1}^{\infty}$ is
an $\ell_{\infty}$-independent sequence of elements of $X$ and $A$ is $\left[
\left(  x_{n}\right)  _{n=1}^{\infty},\mathcal{S}\right]  $-lineable, for some
infinite dimensional subspace $\mathcal{S}$ of $\ell_{\infty}$, then $A$ is lineable.
\end{theorem}

\begin{proof}
Let $\mathcal{S}$ be an infinite dimensional subspace of $\ell_{\infty}$ and
assume that $A$ is $\left[  \left(  x_{n}\right)  _{n=1}^{\infty}%
,\mathcal{S}\right]  $-lineable. Since $\left(  x_{n}\right)  _{n=1}^{\infty}$
is $\ell_{\infty}$-independent in $X$, we can conclude that the operator
$T\colon\mathcal{S}\longrightarrow X$ given by $T\left(  \left(  c_{n}\right)
_{n=1}^{\infty}\right)  =\sum_{n=1}^{\infty}c_{n}x_{n}$ is not only
well-defined and linear, but also injective. Thus,
\[
\dim T\left(  \mathcal{S}\right)  =\dim\mathcal{S}\geq\aleph_{0}\text{.}%
\]
Furthermore, due to the fact that $X$ is Hausdorff, we have
\[
T\left(  \mathcal{S}\right)  \subset A\cup\left\{  0\right\}  \text{.}%
\]
Hence, $A$ is lineable and the proof is complete.
\end{proof}

\section{Negative results\label{section 4}}

In this section, we consider results that provide an enlightening insight into
the $\left[  \mathcal{S}\right]  $-lineability of subsets in some class of
some classes of infinite dimensional vector spaces. The next proposition
highlights the importance of sequence properties concerning their distance
from the origin in normed vector spaces. Roughly speaking, the next statement
shows that sequences away from zero do not generate $\left[  \mathcal{S}%
\right]  $-lineability.

\begin{proposition}
Let $X$ be an infinite dimensional normed space. If $\left(  x_{n}\right)
_{n=1}^{\infty}$ is a linearly independent sequence in $X$ such that
$\inf_{n\in\mathbb{N}}\left\Vert x_{n}\right\Vert _{X}>0$, then for any
infinite dimensional subspace $\mathcal{S}$ of $\ell_{\infty}$ properly
containing $c_{0}$, there is no subset of $X$ that is $\left[  \left(
x_{n}\right)  _{n=1}^{\infty},\mathcal{S}\right]  $-lineable.
\end{proposition}

\begin{proof}
Assume that $A$ is a subset of $X$ which is $\left[  \left(  x_{n}\right)
_{n=1}^{\infty},\mathcal{S}\right]  $-lineable in $X$ for some infinite
dimensional subspace $\mathcal{S}$ of $\ell_{\infty}$ properly containing
$c_{0}$. Given $\left(  c_{n}\right)  _{n=1}^{\infty}\in\mathcal{S}$, we have
\[
\left\Vert c_{n}x_{n}\right\Vert _{X}=\left\vert c_{n}\right\vert \left\Vert
x_{n}\right\Vert _{X}\geq\left\vert c_{n}\right\vert \underset{n\in\mathbb{N}%
}{\inf}\left\Vert x_{n}\right\Vert _{X}\text{ for all }n\in\mathbb{N}\text{.}%
\]
That is,%
\begin{equation}
\left(  \underset{n\in\mathbb{N}}{\inf}\left\Vert x_{n}\right\Vert
_{X}\right)  ^{-1}\cdot\left\Vert c_{n}x_{n}\right\Vert _{X}\geq\left\vert
c_{n}\right\vert \text{ for all }n\in\mathbb{N}\text{.}
\label{a norma ficou limitada por baixo}%
\end{equation}
Due to $($\ref{a norma ficou limitada por baixo}$)$, if we take $\left(
c_{n}\right)  _{n=1}^{\infty}\in$ $\mathcal{S}\setminus c_{0}$, we can infer
that the sequence $\left(  \left\Vert c_{n}x_{n}\right\Vert _{X}\right)
_{n=1}^{\infty}$ does not converges to zero. However, this leads to a
contradiction, since the series $\sum_{n=1}^{\infty}c_{n}x_{n}$ converges in
$X$ for some vector of $A\cup\left\{  0\right\}  $. Therefore, the desired
result follows.
\end{proof}

The result above remains valid in the context of $p$-Banach spaces as well.
Recall that if $p\in(0,1]$, then a $p$-Banach space $X$ is a vector space
equipped with a non-negative function denoted by $\left\Vert \cdot\right\Vert
$, also known as a $p$-norm, which satisfies the following conditions:

\begin{enumerate}
\item $\left\Vert x\right\Vert _{p}=0$ if and only if $x=0$.

\item $\left\Vert tx\right\Vert _{p}=\left\vert t\right\vert ^{p}\left\Vert
x\right\Vert _{p}$ for all $x\in X$ and $t\in\mathbb{R}$.

\item $\left\Vert x+y\right\Vert _{p}\leq\left\Vert x\right\Vert
_{p}+\left\Vert y\right\Vert _{p}$ for all $x,y\in X$.
\end{enumerate}

The space $\ell_{p}$ (with $p\in\left(  0,1\right)  $) of all real sequences
$\left(  x_{n}\right)  _{n=1}^{\infty}$ such that the series $%
{\textstyle\sum_{n=1}^{\infty}}
\left\vert x_{n}\right\vert ^{p}$ converges is an example of $p$-Banach space.
If we consider the sequence $\left(  e_{n}\right)  _{n=1}^{\infty}$ in
$\ell_{\infty}$, where $e_{n}=$ $(0,0,\ldots,0,1,0,0,\ldots)$ (with the $1$ at
the $n$th place), we obtain the following result:

\begin{corollary}
For $p\in\left(  0,1\right)  $ and any infinite dimensional subspace
$\mathcal{S}$ of $\ell_{\infty}$ properly containing $c_{0}$, there is no
subset of $\ell_{p}$ that is $\left[  \left(  e_{n}\right)  _{n=1}^{\infty
},\mathcal{S}\right]  $-lineable.
\end{corollary}

\end{document}